\documentclass[11pt]{amsart}

\usepackage[T1]{fontenc}
\usepackage[utf8]{inputenc}
\usepackage{lmodern}
\usepackage{amsmath,amssymb,mathtools}

\usepackage{xcolor}
\usepackage[colorlinks=true,linkcolor=blue!55!black,citecolor=blue!55!black, urlcolor=blue!65!black]{hyperref}
\usepackage[margin=0.9in]{geometry}

\newtheorem{theorem}{Theorem}[section]
\newtheorem{proposition}[theorem]{Proposition}
\newtheorem{lemma}[theorem]{Lemma}
\newtheorem{corollary}[theorem]{Corollary}
\theoremstyle{definition}
\newtheorem{definition}[theorem]{Definition}
\newtheorem{example}[theorem]{Example}
\theoremstyle{remark}
\newtheorem{remark}[theorem]{Remark}
\newtheorem{question}[theorem]{Question}

\newcommand{\PP}{\mathbb{P}}
\newcommand{\FF}{\mathbb{F}}

\newcommand{\ev}{\mathrm{ev}}
\newcommand{\HF}{\mathrm{HF}}
\newcommand{\ord}{\mathrm{ord}}

\usepackage[overload]{textcase}

\title{Interpolation on Concurrent Lines and the Geproci Property in Positive Characteristic}

\author{}

\author{Luca Chiantini}\address[Chiantini]{Dipartimento di Ingegneria dell'Informazione e Scienze Matematiche,Universit\`a di Siena, Siena, Italy.}\email{luca.chiantini@unisi.it}

\author{{\L}ucja Farnik}\address[Farnik]{Department of Mathematics, University of the National Education Commission, Podchorążych~2, 30-084 Kraków, Poland} \email{lucja.farnik@gmail.com}

\author{Giuseppe Favacchio}\address[Favacchio]{Dipartimento di Ingegneria, Universit\`a degli Studi di Palermo, Viale delle Scienze, 90128 Palermo, Italy}\email{giuseppe.favacchio@unipa.it}

\author{Brian Harbourne} \address[Harbourne]{Department of Mathematics, University of Nebraska--Lincoln, Lincoln, Nebraska 68588,~USA}\email{brianharbourne@unl.edu}

\author{Juan Migliore} \address[Migliore]{Department of Mathematics, University of Notre Dame, Notre Dame, Indiana 46556,~USA} \email{migliore.1@nd.edu}

\author{Tomasz Szemberg} \address[Szemberg]{Department of Mathematics, University of the National Education Commission, Podchorążych~2, 30-084 Kraków, Poland} \email{tomasz.szemberg@gmail.com}

\author{Justyna Szpond} \address[Szpond]{Department of Mathematics, University of the National Education Commission, Podchorążych~2, 30-084 Kraków, Poland} \email{szpond@gmail.com}

\date{}


\begin{document}
\begin{abstract}
We study an interpolation problem for finite point configurations on concurrent lines over fields of positive characteristic, focusing on a natural class arising from the Artin--Schreier map. Given $N$ points on each of a collection of concurrent lines, we ask when they can be simultaneously cut out by a hypersurface of degree $N$ containing none of the lines. We introduce $\mathbb F_N$-Artin--Schreier configurations, for which the interpolation problem reduces to two evaluation conditions on spaces of homogeneous forms. 

This reduction is  useful in the study of the geproci property. Recall a finite set is geproci if it becomes a complete intersection after general projection to a hyperplane. We characterize exactly which $\mathbb F_N$-Artin--Schreier configurations in $\mathbb P^3$ are geproci
and we develop an Artin--Schreier lifting construction that produces geproci sets in $\mathbb P^n$ from geproci sets in $\mathbb P^{n-1}$. This yields, in every positive characteristic, infinitely many not projectively equivalent geproci sets in every $\mathbb P^n$, $n\geq3$, giving the first known geproci sets in dimensions $n\ge4$. No such examples for $n\ge4$ are currently known in characteristic zero.
\end{abstract}
\maketitle

\section{Introduction}

Let $K$ be an algebraically closed field of characteristic $p>0$,  
let $N=p^e$  and let $n\geq2$. Consider a collection of distinct lines $L_1,\ldots,L_q\subset\mathbb P^n$ through a common point, and choose $N$ distinct points on each line. We are interested in the interpolation problem of determining when these point sets can be simultaneously cut out on the lines by a hypersurface of degree $N$ containing none of the $L_i$.  Equivalently, if $f_i\in H^0(L_i,\mathcal O_{L_i}(N))$ cuts out the prescribed points on $L_i$, we ask whether suitable nonzero scalar multiples of the $f_i$ are the restrictions of a single form in $H^0(\mathbb P^n,\mathcal O_{\mathbb P^n}(N))$. The degree $N$ is the smallest possible for such an interpolation problem: a hypersurface of smaller degree containing $N$ distinct points of a line must contain the whole line.

In general, the restrictions of a degree-$N$ form to the different lines are subject to compatibility conditions involving all their coefficients. In positive characteristic, however, there is a natural class of configurations for which these conditions simplify dramatically. The construction is based on the classical Artin--Schreier phenomenon; see \cite{ArtinSchreier1927,StacksAS}. If $\mathbb F_N\subset K$ is the finite field with $N$ elements, then
\[
\prod_{a\in\mathbb F_N}(t-a)=t^N-t \ \in K[t].
\]
Given the concurrent lines $L_1,\ldots,L_q \subset\mathbb P^n$, for each $i$ we can pick $N$ points $U_i\subset L_i$ such that up to a choice of coordinates on $L_i$ the points are the zeros of $t^N-t$ with the point of concurrency being the point at infinity. (Note that we can choose any point $u\in L_i$ away from the point of concurrency to be in $U_i$. Translations by elements of $\mathbb F_N$ then give the other $N-1$ points of $U_i$. Thus the possible choices of $U_i$ are orbits under this action.) We call such a configuration of $qN$ points an $\mathbb F_N$-Artin--Schreier configuration.
The absence of intermediate terms in $t^N-t$ is the key feature of the construction: the problem of interpolating all the orbits $U_i$ by a degree-$N$ hypersurface reduces to two evaluation conditions on spaces of homogeneous forms on the set of directions of the supporting lines; see Proposition~\ref{prop:AS-interpolation}.

For $n\geq 3$, the particular form of Artin--Schreier configurations becomes especially useful when interpolation is combined with projection. Indeed, the Artin--Schreier structure is preserved by projection; see Lemma~\ref{lem:AS-projection}. 

A natural application is to the study of general projections that are complete intersections. A nondegenerate finite set of points in $\mathbb P^n$, $n\geq3$, is said to be \emph{geproci of type $(d_1,\ldots,d_{n-1})$} if its general projection to a hyperplane is a complete intersection of degrees $d_1,\ldots,d_{n-1}$. In $\mathbb P^3$ we simply say $(a,b)$-geproci. Geproci sets were introduced and studied (originally over the complex numbers) in connection with unexpected hypersurfaces and special point configurations; see, for instance, \cite{ChiantiniMigliore2021,POLITUSbook}. 
  Kettinger showed the concept is also of substantial interest in positive characteristics by constructing new kinds of examples of geproci sets,  in $\PP^3_K$; see \cite{Kettinger2024}. Related higher-dimensional phenomena have recently been studied in \cite{POLITUS4,FavacchioKettinger2026}. Nevertheless,  until now  no nondegenerate geproci sets  were known in   $\mathbb P^n$ for $n>3$, and still none are known in characteristic zero.  

Projecting from a general point sends an Artin--Schreier configuration to an Artin--Schreier configuration on the projected lines, with the same parameter on each supporting line. Thus the problem of determining whether a general projection is a complete intersection remains within the same class of interpolation problems. This makes Artin--Schreier configurations particularly well suited to the study of the geproci property.

In $\mathbb P^3$, this leads to a complete characterization; see Theorem~\ref{thm:P3-characterization}: an $\mathbb F_N$-Artin--Schreier configuration supported on $q$ concurrent lines spanning $\mathbb P^3$ is $(q,N)$-geproci if and only if
$q\leq N.$

The sufficiency follows from the interpolation criterion, while the converse requires showing that the interpolation conditions forced by a general projection cannot hold when $q>N$. This is achieved by a projective Vandermonde argument.  In the extremal case $q=N$, the Artin--Schreier structure is not only sufficient but also necessary: a set of $N$ points on each of $N$ concurrent lines spanning $\mathbb P^3$ is $(N,N)$-geproci if and only if the points on each supporting line form an $\mathbb F_N$-Artin--Schreier orbit; see Corollary~\ref{cor:AS-forced}. This is a consequence of the more general rigidity result of Proposition~\ref{prop:forced-sparsity}.

The interpolation viewpoint also provides a mechanism for passing to higher dimension. Starting with a geproci set $Z\subset\mathbb P^n$, we take a cone over $Z$ in $\mathbb P^{n+1}$ and choose an $\mathbb F_N$-Artin--Schreier orbit on each line of the cone.

Since $K$ is algebraically closed, it contains the roots $\mathbb F_{p^e}\subset K$ of  
$t^{p^e}-t$ 
for every $e\geq1$. By choosing $e$ sufficiently large at each step, the surjectivity condition required for the lifting can be satisfied, so the construction can be iterated indefinitely. This yields our main existence result, Theorem~\ref{thm:every-dimension}: over any algebraically closed field of positive characteristic, for every $n\geq3$ there exist infinitely many geproci sets in $\mathbb P^n$.

More generally, starting from any $(d_1,\ldots,d_r)$-geproci set in $\mathbb P^{r+1}$, one may iterate the construction to obtain geproci sets in higher-dimensional projective spaces of types
$(d_1,\ldots,d_r,p^{e_1},\ldots,p^{e_s})$,
where at each step the exponent may be chosen sufficiently large for the required interpolation condition to hold.

The paper is organized as follows. In Section~\ref{sec:AS} we introduce $\mathbb F_N$-Artin--Schreier configurations, study their behavior under projection, and establish the interpolation criterion of Proposition~\ref{prop:AS-interpolation}. In Section~\ref{sec:AS-P3} we apply this criterion to the geproci problem in $\mathbb P^3$ and prove Theorem~\ref{thm:P3-characterization}. In Section~\ref{sec:lifting} we develop the lifting construction, prove Theorem~\ref{thm:geproci-lifting}, and deduce the existence result of Theorem~\ref{thm:every-dimension}. In Section~\ref{sec:rigidity}, we then establish the rigidity results of Proposition~\ref{prop:forced-sparsity} and Corollary~\ref{cor:AS-forced}, before concluding with some open questions in Section~\ref{sec:questions}.

The construction is intrinsically tied to positive characteristic: the Artin--Schreier structure that makes both the interpolation and the lifting possible has no direct counterpart in characteristic zero. The existence of geproci sets in $\mathbb P^n$ for $n\geq4$ in characteristic zero remains an open problem.

\section{Artin--Schreier configurations and interpolation}\label{sec:AS}

Throughout this section, let $K$ be an algebraically closed field of characteristic $p>0$, and let $N=p^e$, with $e\geq1$. The basic identity underlying our construction is
\[
\prod_{a\in\FF_N}(t-a)=t^N-t.
\]
Thus, if $\beta,\lambda\in K$, with $\lambda\neq0$, the affine copy $B=\beta+\lambda\FF_N$ is the zero set of $t^N-\mu t+\nu$, where $\mu=\lambda^{N-1}$ and $\nu=\lambda^{N-1}\beta-\beta^N$. Indeed,
\[
\prod_{a\in\FF_N}\bigl(t-(\beta+\lambda a)\bigr)
=(t-\beta)^N-\lambda^{N-1}(t-\beta)
=t^N-\mu t+\nu.
\]
The absence of intermediate terms is the key feature that makes the interpolation method effective.

\begin{definition}
\label{def:AS-configuration}
Let $L_1,\ldots,L_q\subset\PP^n$ be distinct lines through a common point $O$. An \emph{$\FF_N$-Artin--Schreier configuration} supported on $L_1,\ldots,L_q$ is a finite set $Z=B_1\cup\cdots\cup B_q$, where $B_r\subset L_r\setminus\{O\}$ is, in a suitable affine coordinate on $L_r$, an affine copy $B_r=\beta_r+\lambda_r\FF_N$, with $\lambda_r\neq0$. We call $B_r$ the {\em Artin--Schreier orbit} on $L_r$.
\end{definition}

Thus every orbit consists of $N$ distinct points. Notice also that the definition is independent of the choice of affine coordinate fixing $O$, since an affine transformation sends an affine copy of $\FF_N$ to another affine copy of $\FF_N$.

This observation immediately gives the behavior under projection.

\begin{lemma}
\label{lem:AS-projection}
Let $Z\subset\PP^n$ be an $\FF_N$-Artin--Schreier configuration supported on distinct concurrent lines $L_1,\ldots,L_q$, and let $\pi_P\colon\PP^n\dashrightarrow H\simeq\PP^{n-1}$ be projection from a point $P$ which does not belong to any $L_r$. Assume that the lines $\pi_P(L_r)$ are distinct. Then $\pi_P(Z)$ is an $\FF_N$-Artin--Schreier configuration supported on the concurrent lines $\pi_P(L_r)$.
\end{lemma}

\begin{proof}
The restriction of $\pi_P$ to each $L_r$ is a projectivity onto $\pi_P(L_r)$, sending the common point $O$ to the common point $\pi_P(O)$. After removing these points, the induced map between the corresponding affine lines is affine. It therefore sends an affine copy of $\FF_N$ to another affine copy of $\FF_N$.
\end{proof}

We now describe the interpolation problem associated with such a configuration. Let $H\simeq\PP^{n-1}$ be a hyperplane not containing $O$, and write $R_r=L_r\cap H$. The set $\Gamma=\{R_1,\ldots,R_q\}\subset H$ will be called the \emph{direction set} of the supporting lines.

Choose homogeneous coordinates $[x_0:x_1:\cdots:x_n]$ such that $O=[1:0:\cdots:0]$ and $H=V(x_0)$. After choosing a representative $v_r$ of $R_r$, every point of $L_r\setminus\{O\}$ can be written uniquely as $[t:v_r]$. In this coordinate the orbit $B_r$ is cut out by $t^N-\mu_rt+\nu_r$, with $\mu_r\neq0$.

Let $S=K[x_1,\ldots,x_n]$. For $d\geq0$, denote by
\[
\ev_d\colon S_d\longrightarrow K^q,\qquad
F\longmapsto\bigl(F(v_1),\ldots,F(v_q)\bigr)
\]
the evaluation map. The choice of the representatives $v_r$ amounts to choosing trivializations at the points $R_r$; changing them rescales the corresponding coefficients coherently.

The following criterion is the basic interpolation result.

\begin{proposition}
\label{prop:AS-interpolation}
With the notation above, there exists a hypersurface $V\subset\PP^n$ of degree $N$ such that $V\cap L_r=B_r$ scheme-theoretically for every $r$ if and only if
\begin{equation}
\label{eq:AS-interpolation}
(\mu_1,\ldots,\mu_q)\in\operatorname{Im}(\ev_{N-1})
\quad\text{and}\quad
(\nu_1,\ldots,\nu_q)\in\operatorname{Im}(\ev_N).
\end{equation}
In this case, $V$ contains none of the lines $L_r$.
\end{proposition}

\begin{proof}
Let $F$ be a homogeneous form of degree $N$. Its restriction to $L_r$, in the coordinate $[t:v_r]$, has the form
\[
F(t,v_r)=
c_0t^N+c_1(v_r)t^{N-1}+\cdots+c_{N-1}(v_r)t+c_N(v_r),
\]
where $c_i$ is a homogeneous form of degree $i$ in $x_1,\ldots,x_n$.

Suppose first that $F$ cuts out $B_r$ on every $L_r$. Since $B_r$ consists of the $N$ roots of $t^N-\mu_rt+\nu_r$, the two degree-$N$ polynomials are proportional. Since $O\notin B_r$ and $F|_{L_r}$ cuts out $B_r$ scheme-theoretically, $F(O)\neq0.$ Since the leading coefficient $c_0=F(O)\neq0$ is independent of $r$, after rescaling $F$, we may assume $c_0=1$. It follows that $c_i(v_r)=0$ for $1\leq i\leq N-2$, $c_{N-1}(v_r)=-\mu_r$, and $c_N(v_r)=\nu_r$. Hence the two vectors in \eqref{eq:AS-interpolation} belong to the indicated evaluation images.

Conversely, suppose that \eqref{eq:AS-interpolation} holds. Choose $M\in S_{N-1}$ and $Q\in S_N$ such that $M(v_r)=\mu_r$ and $Q(v_r)=\nu_r$ for every $r$. Then
\begin{equation}
\label{eq:AS-hypersurface}
F=x_0^N-x_0M+Q
\end{equation}
is homogeneous of degree $N$, and its restriction to $L_r$ is $t^N-\mu_rt+\nu_r$. Thus $V(F)\cap L_r=B_r$ scheme-theoretically. Since this restriction is not identically zero, $V(F)$ contains none of the supporting lines.
\end{proof}

\begin{remark}
\label{rem:coherent-rescaling}
The coefficient vectors in Proposition~\ref{prop:AS-interpolation} depend on the choice of representatives of the directions, but the criterion does not. Replacing $v_r$ by $c_rv_r$ changes the affine coordinate on $L_r$, and hence changes $\mu_r$ and $\nu_r$ by precisely the homogeneous rescalings occurring in evaluation in degrees $N-1$ and $N$, respectively.
\end{remark}

For later use, we record the standard maximal-rank consequence.

\begin{definition}
\label{def:d-general-position}
A finite set $\Gamma\subset\PP^s$ of $q$ points is in
\emph{$d$-general position} if
\[
\dim_K(S/I_\Gamma)_d=\min\{q,\dim_K S_d\}.
\]
Equivalently, the evaluation map in degree $d$ has maximal rank.
\end{definition}

\begin{corollary}
\label{cor:AS-general-position}
Let $Z$ be an $\FF_N$-Artin--Schreier configuration supported on $q$ concurrent lines with direction set $\Gamma\subset\PP^{n-1}$. Suppose that $\Gamma$ is in $(N-1)$-general position and that $q\leq\dim_K S_{N-1}$.
Then there exists a hypersurface of degree $N$ containing $Z$ and none of its supporting lines.
\end{corollary}

\begin{proof}
The assumptions imply that $\ev_{N-1}$ is surjective. The evaluation map in degree $N$ is then surjective as well: choose a linear form $\ell$ which does not vanish at any point of~$\Gamma$, and multiply interpolating forms of degree $N-1$ by $\ell$, after rescaling the prescribed values. Proposition~\ref{prop:AS-interpolation} now applies.
\end{proof}

The interpolation criterion also gives a simple numerical sufficient condition for an Artin--Schreier configuration to be geproci.

\begin{corollary}
\label{cor:higher-dimensional-sufficient}
Let $Z\subset\PP^n$, $n\geq3$, be an
$\FF_N$-Artin--Schreier configuration supported on $q$ concurrent
lines spanning $\PP^n$, where $N=p^e$. Suppose that, for a general projection
$\pi_P\colon\PP^n\dashrightarrow\PP^{n-1}$,

\begin{enumerate}
\item the union of the projected lines is a reduced
complete intersection of type $(d_1,\ldots,d_{n-2})$;
\item the direction set
$\Gamma_P\subset\PP^{n-2}$ is in $(N-1)$-general position;
\item
$q\leq
h^0\bigl(\PP^{n-2},
\mathcal O_{\PP^{n-2}}(N-1)\bigr)
=
\binom{N+n-3}{n-2}.$
\end{enumerate}

Then $Z$ is geproci. Precisely, its general projection is a
complete intersection of type
$(d_1,\ldots,d_{n-2},N).$
\end{corollary}

\begin{proof}
By Lemma~\ref{lem:AS-projection}, $\pi_P(Z)$ is an
$\FF_N$-Artin--Schreier configuration on the projected supporting
lines. Conditions (2) and (3), together with
Corollary~\ref{cor:AS-general-position}, give a hypersurface of
degree $N$ containing $\pi_P(Z)$ and none of its supporting lines.

By (1), the union of the projected supporting lines is a complete
intersection of type $(d_1,\ldots,d_{n-2})$. Intersecting it with
the degree-$N$ hypersurface gives precisely the $qN$ reduced points
of $\pi_P(Z)$. Thus $\pi_P(Z)$ is a complete intersection of type
$(d_1,\ldots,d_{n-2},N)$, and hence $Z$ is geproci.
\end{proof}
\begin{example}
\label{ex:P4-twelve-points}
Let $K$ be algebraically closed of characteristic $3$. In $\PP^4$
with coordinates $[x_0:x_1:x_2:x_3:x_4]$, set
$O=[1:0:0:0:0]$ and, for $i=1,\ldots,4$, let
\[
L_i=\{[t:e_i]\mid t\in K\}\cup\{O\},
\qquad
A_i=\{[j:e_i]\mid j\in\FF_3\}.
\]
Then 
$Z=A_1\cup A_2\cup A_3\cup A_4$ 
is a $(2,2,3)$-geproci set of $12$ points in $\PP^4$.

Indeed, the four lines are concurrent and span $\PP^4$. After a
general projection to $\PP^3$, their directions are four points in
$\PP^2$ with no three collinear. Hence they form a complete
intersection of type $(2,2)$ and are in $2$-general position. Since 
$4\leq h^0(\PP^2,\mathcal O_{\PP^2}(2))=6,$ 
Corollary~\ref{cor:higher-dimensional-sufficient} applies.

Explicitly, the twelve points are
\[
\begin{array}{lll}
{}[0:1:0:0:0],&[1:1:0:0:0],&[2:1:0:0:0],\\
{}[0:0:1:0:0],&[1:0:1:0:0],&[2:0:1:0:0],\\
{}[0:0:0:1:0],&[1:0:0:1:0],&[2:0:0:1:0],\\
{}[0:0:0:0:1],&[1:0:0:0:1],&[2:0:0:0:1].
\end{array}
\]
\end{example}

\section{Artin--Schreier geproci sets in \texorpdfstring{$\PP^3$}{P3}}
\label{sec:AS-P3}

We now specialize the results of the previous section to $\PP^3$. In this case the interpolation problem becomes particularly transparent, since the directions of the projected supporting lines form a finite set in $\PP^1$.

Let $N=p^e$, and let $Z\subset\PP^3$ be an $\FF_N$-Artin--Schreier configuration supported on $q$ distinct lines $L_1,\ldots,L_q$ through a point $O$. Throughout this section we assume that the supporting lines span $\PP^3$.

For a general point $P\in\PP^3$, the projection $\pi_P(Z)\subset\PP^2$ is supported on $q$ distinct concurrent lines. Their union is a plane curve $G_P$ of degree $q$ containing $\pi_P(Z)$. We first obtain a sufficient condition for $Z$ to be $(q,N)$-geproci.

\begin{proposition}
\label{prop:P3-sufficient}
Let $Z\subset\PP^3$ be an $\FF_N$-Artin--Schreier configuration supported on $q$ concurrent lines spanning $\PP^3$. If $q\leq N$, then $Z$ is a $(q,N)$-geproci set.
\end{proposition}

\begin{proof}
After projection from a general point, the directions of the $q$ supporting lines form a set $\Gamma$ of $q$ distinct points in $\PP^1$. Since $h^0(\PP^1,\mathcal O_{\PP^1}(N-1))=N$, the assumption $q\leq N$ gives $\HF_\Gamma(N-1)=q$. Hence the evaluation map in degree $N-1$ is surjective, and the same is true in degree $N$.
Proposition~\ref{prop:AS-interpolation} therefore gives a plane curve $F_P$ of degree $N$ containing $\pi_P(Z)$ and none of the supporting lines.

Thus $F_P$ and $G_P$ have no common component. Since they have degrees $N$ and $q$ and contain the $Nq$ reduced points of $\pi_P(Z)$, B\'ezout's theorem gives $\pi_P(Z)=V(F_P,G_P)$. Hence $Z$ is $(q,N)$-geproci.
\end{proof}

We shall prove that the inequality $q\leq N$ is also necessary. The argument uses the following elementary lemma.

\begin{lemma}
\label{lem:Vandermonde-independence}
Let $v_1,\ldots,v_s$ be distinct points of $\PP^2$, not all collinear. For $1\leq i\leq s$, define the homogeneous polynomial
\[
\Phi_i(u)=
\prod_{\substack{j<k\\ j,k\neq i}}
[u,v_j,v_k],
\]
where $[u,v_j,v_k]$ denotes the determinant of homogeneous representatives of $u,v_j,v_k$. Then $\Phi_1,\ldots,\Phi_s$ are linearly independent.
\end{lemma}

\begin{proof}
Suppose that $\sum_{i=1}^s c_i\Phi_i=0$. We show that each $c_i$ is zero.

Fix an index $i$. Since the points $v_1,\ldots,v_s$ are not all collinear, we can choose a line $M$ through $v_i$ and at least one other $v_j$, but not containing all the $v_h$. Let $S=\{j\mid v_j\in M\}$, $r=|S|$, and $a=s-r>0$. We denote also by $M$ a linear equation of this line.

If $h\in S$, the product defining $\Phi_h$ contains the factor $M$ once for every pair of points in $S\setminus\{h\}$, and therefore $\ord_M(\Phi_h)=\binom{r-1}{2}$. If $h\notin S$, all pairs of points in $S$ occur, and hence $\ord_M(\Phi_h)=\binom r2$.

Divide the relation by $M^{\binom{r-1}{2}}$ and restrict it to $M$. Since $\binom r2-\binom{r-1}{2}=r-1>0$, all terms indexed by $h\notin S$ vanish on $M$.

Consider now a term with $h\in S$. After division by $M^{\binom{r-1}{2}}$, the factors corresponding to pairs of points in $S\setminus\{h\}$ contribute only a nonzero scalar. The factors corresponding to pairs of points outside $S$ are common to all these terms. The remaining factors come from pairs $(v_j,v_k)$ with $j\in S\setminus\{h\}$ and $k\notin S$. The line $\langle v_j,v_k\rangle$ meets $M$ at $v_j$, so its restriction to $M$ is a nonzero scalar multiple of a linear form $\ell_j$ vanishing at $v_j$. Since there are exactly $a$ choices for $k\notin S$, after removing a common nonzero polynomial factor on $M$ the restricted relation has the form
\[
\sum_{h\in S}d_h
\prod_{\substack{j\in S\\j\neq h}}\ell_j^{\,a}=0,
\]
where $d_h$ is a nonzero scalar multiple of $c_h$. Evaluate this identity at $v_i$. If $h\neq i$, the corresponding product contains $\ell_i^a$ and therefore vanishes at $v_i$. The term with $h=i$, on the other hand, does not contain $\ell_i$ and is nonzero at $v_i$, since the points of $S$ are distinct. Thus $d_i=0$, and hence $c_i=0$. Since $i$ was arbitrary, all the coefficients $c_i$ vanish.
\end{proof}

\begin{example}
\label{ex:Vandermonde-independence}
We illustrate the mechanism of the proof in a simple case. Suppose that
$v_1,v_2,v_3$ lie on a line $M$, while $v_4\notin M$. Then
\[
\begin{aligned}
\Phi_1&=[u,v_2,v_3][u,v_2,v_4][u,v_3,v_4],\qquad 
\Phi_2=[u,v_1,v_3][u,v_1,v_4][u,v_3,v_4],\\
\Phi_3&=[u,v_1,v_2][u,v_1,v_4][u,v_2,v_4],\qquad 
\Phi_4=[u,v_1,v_2][u,v_1,v_3][u,v_2,v_3].
\end{aligned}
\]
The factors determined by two of $v_1,v_2,v_3$ are scalar multiples
of the equation of $M$. Hence
\[
\ord_M(\Phi_1)=\ord_M(\Phi_2)=\ord_M(\Phi_3)=1,
\qquad
\ord_M(\Phi_4)=3.
\]
Thus, starting from a relation
\[
c_1\Phi_1+c_2\Phi_2+c_3\Phi_3+c_4\Phi_4=0,
\]
dividing by $M$ and restricting to $M$ eliminates the last term.
Up to nonzero scalar factors, the resulting relation on $M$ is
\[
d_1\ell_2\ell_3+
d_2\ell_1\ell_3+
d_3\ell_1\ell_2=0,
\]
where $\ell_j$ is a linear form on $M$ vanishing at $v_j$.
Evaluating at $v_1$ kills the last two terms, while
$\ell_2(v_1)\ell_3(v_1)\neq0$, and therefore $d_1=0$, hence
$c_1=0$. Evaluating similarly at $v_2$ and $v_3$ gives
$c_2=c_3=0$, and the original relation then gives $c_4=0$.

This is precisely the mechanism used in the general proof: the
order of vanishing along $M$ first separates the points on $M$
from those outside it, and evaluation at the individual points
of $M$ then separates the remaining terms.
\end{example}
The preceding argument may be viewed as a projective analogue of the usual Vandermonde separation argument: after restriction to $M$, evaluation at the points $v_i$ isolates the individual terms.

We can now characterize all $\FF_N$-Artin--Schreier configurations of this type in $\PP^3$.

\begin{theorem}
\label{thm:P3-characterization}
Let $N=p^e$, and let $Z\subset\PP^3$ be an
$\FF_N$-Artin--Schreier configuration supported on $q$
concurrent lines spanning $\PP^3$. Then $Z$ is
$(q,N)$-geproci if and only if $q\leq N$.
\end{theorem}

\begin{proof}
If $q\leq N$, the conclusion follows from Proposition~\ref{prop:P3-sufficient}. For the converse, assume that $Z$ is $(q,N)$-geproci and suppose that $q>N$.

For a general center of projection $P$, let $F_P$ be a general element of $[I(\pi_P(Z))]_N$, and let $G_P$ be the union of the $q$ projected supporting lines. The curve $F_P$ contains none of these lines. Indeed, if a supporting line $L$ were a component of $F_P$, then the other member of the complete intersection could not contain $L$, and its intersection with $L$ would have length $q>N$, whereas $\pi_P(Z)\cap L$ consists of only $N$ reduced points.

Thus, on each supporting line $L$, the curve $F_P$ contains the $N$ points of $\pi_P(Z)\cap L$ and does not contain $L$. By B\'ezout, these are precisely the points of $F_P\cap L$. Therefore $\pi_P(Z)=V(G_P,F_P).$

The curve $F_P$ cannot pass through the common point of the supporting lines. Otherwise, on each supporting line it would contain that point together with the $N$ points of $\pi_P(Z)$; B\'ezout's theorem would force every supporting line to be a component of $F_P$, which is impossible because $q>N$.
Hence $F_P$ satisfies the hypotheses of Proposition~\ref{prop:AS-interpolation}.

Since the supporting lines span $\PP^3$, their directions are not all collinear in $\PP^2$. Since $q>N$, we may therefore choose $N+1$ supporting lines whose directions $v_1,\ldots,v_{N+1}\in\PP^2$ are not all collinear. Write the Artin--Schreier equation on the $i$-th line as $t^N-\mu_i t+\nu_i=0$, with $\mu_i\neq0$.

Choose coordinates with $O=[1:0:0:0]$, write the $i$-th supporting line as $[t:v_i]$, and take a general center of projection $P=[1:u]$. Projection onto the plane $x_0=0$ sends $[t:v_i]$ to $[v_i-tu]$. Choose a complement $K^3=Ku\oplus W$ and write $v_i=\alpha_i u+w_i$, with $w_i\in W$. Then the projected line is parametrized by $[w_i+(\alpha_i-t)u]$. Setting $s=\alpha_i-t$ and using $N=p^e$, the equation $t^N-\mu_i t+\nu_i=0$ becomes, after multiplication by a nonzero scalar, $s^N-\mu_i s+\nu_i'=0$ for a suitable $\nu_i'\in K$. Thus the coefficient $\mu_i$ is unchanged by projection, while the projected direction is represented by $w_i\in W$.

By Proposition~\ref{prop:AS-interpolation}, for every general $u$ the vector 
$(\mu_1,\ldots,\mu_{N+1})$ 
belongs to the image of the evaluation map
\[
H^0\bigl(\PP^1,\mathcal O_{\PP^1}(N-1)\bigr)
\longrightarrow K^{N+1}
\]
at the $N+1$ projected directions.

Since 
$h^0\bigl(\PP^1,\mathcal O_{\PP^1}(N-1)\bigr)=N,$
this condition is equivalent to the vanishing of the augmented Vandermonde determinant. Expanding along the column $(\mu_i)$ gives, up to an overall nonzero scalar,
\begin{equation}
\label{eq:P3-Vandermonde-relation}
\sum_{i=1}^{N+1}(-1)^i\mu_i
\prod_{\substack{j<k\\j,k\neq i}}
[u,v_j,v_k]=0.
\end{equation}
Indeed, the homogeneous Vandermonde determinant obtained by omitting the $i$-th projected direction is the product of the pairwise determinants, and these are, up to a common nonzero scalar, the determinants $[u,v_j,v_k]$.

Equation~\eqref{eq:P3-Vandermonde-relation} holds for every general center $u$, hence it is a polynomial identity in $u$. By Lemma~\ref{lem:Vandermonde-independence}, the polynomials 
$\Phi_i(u)=
\prod_{\substack{j<k\\j,k\neq i}}
[u,v_j,v_k]$
are linearly independent. Therefore all the coefficients in \eqref{eq:P3-Vandermonde-relation} vanish, so
$(-1)^i\mu_i=0$
 for every $i$. 
Thus $\mu_i=0$ for every $i$, contradicting $\mu_i\neq0$.
Therefore $q\leq N$.
\end{proof}

\begin{remark}
The Vandermonde relation appearing in the proof is especially transparent when $N=2$, when we consider the $N+1=3$ supporting lines. In this case, for three projected directions $w_i=[a_i:b_i]\in\PP^1,$ the condition 
$(\mu_1,\mu_2,\mu_3)\in\operatorname{Im}(\ev_1)$
is equivalent to
\[
\det
\begin{pmatrix}
a_1 & b_1 & \mu_1\\
a_2 & b_2 & \mu_2\\
a_3 & b_3 & \mu_3
\end{pmatrix}=0.
\]
Expanding along the last column gives
\[
\mu_1(a_2b_3-a_3b_2)
-\mu_2(a_1b_3-a_3b_1)
+\mu_3(a_1b_2-a_2b_1)=0.
\]
Up to a common nonzero scalar, the pairwise determinants are $[u,v_j,v_k]$, and hence this becomes
\[
\mu_1[u,v_2,v_3]
-\mu_2[u,v_1,v_3]
+\mu_3[u,v_1,v_2]=0.
\]
Thus the general relation
\eqref{eq:P3-Vandermonde-relation} is exactly the higher-degree homogeneous Vandermonde analogue of this $3\times3$ determinant.
\end{remark}

\begin{example}
\label{ex:P3-family}
Let $N=p^e$. For every $3\leq q\leq N$, choose $q$ concurrent distinct lines spanning $\PP^3$ and an $\FF_N$-Artin--Schreier orbit on each line. Theorem~\ref{thm:P3-characterization} gives a $(q,N)$-geproci set of $qN$ points. In particular, this construction works in characteristic $2$: taking $N=4$ gives $(3,4)$- and $(4,4)$-geproci sets of $12$ and $16$ points, respectively.
\end{example}

\section{Lifting and higher-dimensional constructions}
\label{sec:lifting}

We now use Artin--Schreier interpolation to construct geproci sets in
higher-dimensional projective spaces. The basic operation is a cone
construction.

Let $X\subset\PP^{n-1}$ be a finite reduced set, embedded in a
hyperplane of $\PP^n$, and choose a point
$O\in\PP^n\setminus\PP^{n-1}$. For each $Q\in X$, let
$L_Q=\langle O,Q\rangle$. The union of the $L_Q$ is the cone
$C_O(X)$ over $X$ with vertex $O$.

Fix $N=p^e$. On each generator $L_Q$, choose an
$\FF_N$-Artin--Schreier orbit of $N$ points. The following
observation relates the geometry of the cone to the interpolation
criterion of Section~\ref{sec:AS}.

\begin{proposition}
\label{prop:AS-lifting}
Let $X\subset\PP^{n-1}$ be a reduced complete intersection of type
$(d_1,\ldots,d_{n-1})$, and let $Z\subset\PP^n$ be obtained by
placing an $\FF_N$-Artin--Schreier orbit on each generator of
$C_O(X)$. Assume that the evaluation map
\[
H^0(\PP^{n-1},\mathcal O_{\PP^{n-1}}(N-1))
\longrightarrow K^{|X|}
\]
is surjective. Then $Z$ is a complete intersection of type
$(d_1,\ldots,d_{n-1},N)$.
\end{proposition}

\begin{proof}
Let $F_1,\ldots,F_{n-1}$ be homogeneous equations defining $X$ in $\PP^{n-1}$. Viewing the $F_i$ as forms on $\PP^n$ which do not involve the coordinate of the vertex, their common zero locus is $C_O(X)$. Hence the cone is a complete intersection of type $(d_1,\ldots,d_{n-1})$.

The assumed surjectivity in degree $N-1$ also implies surjectivity in degree $N$. Indeed, choose a linear form $\ell$ which does not vanish at any point of $X$. Given arbitrary values $(a_Q)_{Q\in X}$, surjectivity in degree $N-1$ gives a form $M$ such that
\[
M(Q)=\frac{a_Q}{\ell(Q)}
\qquad\text{for every }Q\in X.
\]
Then $\ell M$ has degree $N$ and satisfies
$(\ell M)(Q)=a_Q$ for every $Q\in X$.

Hence both evaluation conditions of
Proposition~\ref{prop:AS-interpolation} are satisfied, and there exists a hypersurface $F$ of degree $N$ containing all the prescribed Artin--Schreier orbits and none of the generators. Moreover, its restriction to every generator cuts out exactly the prescribed orbit. Consequently $Z=C_O(X)\cap V(F)$ scheme-theoretically, and $Z$ is a complete intersection of type $(d_1,\ldots,d_{n-1},N)$.
\end{proof}

For geproci sets, the direction set appearing after a general projection is itself a general projection of the set from which the cone was constructed. This gives the lifting theorem.

\begin{theorem}
\label{thm:geproci-lifting}
Let $X\subset\PP^n$ be a geproci set whose general projection to $\PP^{n-1}$ is a complete intersection of type $(d_1,\ldots,d_{n-1})$. Let $N=p^e$. Take a cone $C_O(X)$ over $X$ in $\PP^{n+1}$, and place an $\FF_N$-Artin--Schreier orbit on each
generator. Let $Z \subset \PP^{n+1}$ be the union of these orbits.

Let $Y$ be a general projection of $X$ to $\PP^{n-1}$. If the evaluation map on $Y$ is surjective in degree $N-1$, then $Z$ is a geproci set whose general projection is a complete intersection of type $(d_1,\ldots,d_{n-1},N)$.
\end{theorem}

\begin{proof}
Let $\pi_P\colon\PP^{n+1}\dashrightarrow\PP^n$ be a general projection, and set $R=\langle O,P\rangle\cap\PP^n.$ The generators of the cone project to distinct lines through $\pi_P(O)$. Projecting these lines from $\pi_P(O)$ is equivalent to projecting $X$ from $R$. Since $P$ is general, so is $R$; hence their directions form a general projection $\Gamma\subset\PP^{n-1}$ of $X$. By hypothesis, $\Gamma$ is a complete intersection of type $(d_1,\ldots,d_{n-1})$.

By Lemma~\ref{lem:AS-projection}, the $N$ points on each generator project to an $\FF_N$-Artin--Schreier orbit on the corresponding line. The union of the projected supporting lines is the cone over $\Gamma$. By the assumed surjectivity in degree $N-1$, Proposition~\ref{prop:AS-lifting} applies and shows that $\pi_P(Z)$ is a complete intersection of type $(d_1,\ldots,d_{n-1},N).$ Thus the general projection of $Z$ is a complete intersection.

Finally, $Z$ is nondegenerate. The generators span $\PP^{n+1}$, since $X$ is nondegenerate in $\PP^n$ by the definition of geproci sets. Moreover, each generator contains $N\geq2$ distinct points of $Z$, and hence is contained in the linear span of $Z$. Therefore $Z$ spans $\PP^{n+1}$.
\end{proof}

\begin{example}
\label{ex:P4-four-points}
Let $X\subset\PP^3$ consist of four general points. Then $X$ is $(2,2)$-geproci. Its general projection consists of four general points in $\PP^2$. Hence, for every power $N=p^e$ with $N\geq3$, the evaluation map in degree $N-1$ is surjective.

Placing an $\FF_N$-Artin--Schreier orbit on each of the four generators of a cone over $X$ gives a geproci set of $4N$ points in $\PP^4$, whose general projection is a complete intersection of type $(2,2,N)$. For $N=3$, this recovers the twelve-point $(2,2,3)$-geproci configuration of Example~\ref{ex:P4-twelve-points}, now viewed as the first instance of the lifting construction.

Notice that the starting configuration need not itself be an Artin--Schreier configuration.
\end{example}

\begin{remark}
\label{rem:why-higher-AS}
The use of extension-field Artin--Schreier orbits is essential to
the iteration. Replacing a single $\FF_{p^e}$-orbit by an arbitrary
union of ordinary $\FF_p$-orbits would produce a polynomial with
many intermediate coefficients, and hence a substantially more
complicated interpolation problem.

For an affine copy of $\FF_{p^e}$, on the other hand, the identity
$\prod_{a\in\FF_{p^e}}(t-a)=t^{p^e}-t$
shows that the defining polynomial still has only the two
nonleading coefficients appearing in
Proposition~\ref{prop:AS-interpolation}. At the same time, by
passing to sufficiently large extension fields, the interpolation
degree can be made arbitrarily large, which is what makes the
iteration possible.
\end{remark}

We now iterate the lifting construction. The only numerical input is
the regularity of a zero-dimensional complete intersection. If
$\Gamma\subset\PP^r$ is a complete intersection of type
$(d_1,\ldots,d_r)$, then
\[
\operatorname{reg}(S/I_\Gamma)
=\sum_{i=1}^r(d_i-1).
\]
Thus the evaluation map on $\Gamma$ is surjective in every degree
at least $\operatorname{reg}(S/I_\Gamma)$.

We first treat characteristics different from $2$.

\begin{proposition}
\label{prop:every-dimension-odd}
Let $K$ be an algebraically closed field of characteristic $p\geq3$. For every $n\geq3$, there exists a
geproci set of
$3p^{(n-1)(n-2)/2}$
points in $\PP^n$ whose general projection is a complete
intersection of type
$(3,p,p^2,\ldots,p^{n-2}).$
\end{proposition}

\begin{proof}
For $n=3$, choose three concurrent non-coplanar lines and an
$\FF_p$-Artin--Schreier orbit on each. By
Theorem~\ref{thm:P3-characterization}, this gives a $(3,p)$-geproci set. Suppose that $Z_m\subset\PP^m$ has been constructed and that its
general projection $\Gamma_m\subset\PP^{m-1}$ is a complete
intersection of type 
$(3,p,p^2,\ldots,p^{m-2}).$ 
Its coordinate ring has regularity
$\displaystyle
\operatorname{reg}(S/I_{\Gamma_m})
=
2+\sum_{i=1}^{m-2}(p^i-1).$ 

For $p\geq3$,\quad 
 $\displaystyle p^{m-1}-3-\sum_{i=1}^{m-2}p^i
=
\frac{(p-2)p^{m-1}-2p+3}{p-1}>0.$  

Hence $\displaystyle
\operatorname{reg}(S/I_{\Gamma_m})
=
2+\sum_{i=1}^{m-2}(p^i-1)
\leq
2+\sum_{i=1}^{m-2}p^i
<
p^{m-1}-1.
$

Thus the evaluation map on $\Gamma_m$ is surjective in degree
$p^{m-1}-1$. Take the cone over $Z_m$ in $\PP^{m+1}$ and place an
$\FF_{p^{m-1}}$-Artin--Schreier orbit on each generator.
Theorem~\ref{thm:geproci-lifting} gives a geproci set $Z_{m+1}$ whose general projection is a complete intersection
of type
$(3,p,p^2,\ldots,p^{m-1}).$
This completes the induction.

Finally,
$|Z_n|
=
3p^{1+2+\cdots+(n-2)}
=
3p^{(n-1)(n-2)/2}.$
\end{proof}

In characteristic $2$, the same construction starts with $\FF_4$ rather than $\FF_2$.

\begin{proposition}
\label{prop:every-dimension-char2}
Let $K$ be an algebraically closed field of characteristic $2$. For every $n\geq3$, there exists a geproci set of $3\cdot2^{n(n-1)/2-1}$ points in $\PP^n$ whose general projection is a complete intersection of type $(3,4,8,\ldots,2^{n-1}).$
\end{proposition}

\begin{proof}
For $n=3$, choose three concurrent non-coplanar lines and an $\FF_4$-Artin--Schreier orbit on each. By Theorem~\ref{thm:P3-characterization}, this gives a $(3,4)$-geproci set.

Suppose that $Z_m\subset\PP^m$ has been constructed and that its general projection $\Gamma_m\subset\PP^{m-1}$ is a complete intersection of type
$(3,4,8,\ldots,2^{m-1}).$
Its coordinate ring has regularity
\[
\operatorname{reg}(S/I_{\Gamma_m})
=
2+\sum_{i=2}^{m-1}(2^i-1)
=
2^m-m
\leq 2^m-1.
\]
Hence the evaluation map on $\Gamma_m$ is surjective in degree $2^m-1$.

Take the cone over $Z_m$ in $\PP^{m+1}$ and place an $\FF_{2^m}$-Artin--Schreier orbit on each generator. Theorem~\ref{thm:geproci-lifting} gives a geproci set $Z_{m+1}$ whose general projection is a complete intersection of type $(3,4,8,\ldots,2^m).$
This completes the induction. 
Finally, 
$|Z_n|
=
3\cdot2^{2+3+\cdots+(n-1)}
=
3\cdot2^{n(n-1)/2-1}.$
\end{proof}

Combining the two constructions gives the main existence result.

\begin{theorem}
\label{thm:every-dimension}
Let $K$ be an algebraically closed field of positive characteristic. For every $n\geq3$, there exist infinitely many geproci sets in $\PP^n$.
\end{theorem}
\begin{proof}
For $n=3$, the assertion follows from Example~\ref{ex:P3-family}, by allowing $N=p^e$ to vary.

Let $n\geq4$. Propositions~\ref{prop:every-dimension-odd} and~\ref{prop:every-dimension-char2} provide, according to the characteristic, a geproci set $X\subset\PP^{n-1}$. The evaluation map on a general projection of $X$ is surjective in every sufficiently large degree. Hence Theorem~\ref{thm:geproci-lifting} applies with $N=p^e$ for every sufficiently large $e$. Varying $e$ gives geproci sets in $\PP^n$ with distinct complete-intersection types, and hence infinitely many geproci sets.
\end{proof}

\begin{example}
\label{ex:P4-char2}
In characteristic $2$, start with a $(3,4)$-geproci
$\FF_4$-Artin--Schreier configuration
$Z_3\subset\PP^3$ of twelve points. Take a cone over $Z_3$ in
$\PP^4$, and place an $\FF_8$-Artin--Schreier orbit on each of
its twelve generators.

After a general projection to $\PP^3$, the directions of the
supporting lines form a general projection of $Z_3$, hence a
complete intersection of type $(3,4)$ in $\PP^2$. Its coordinate
ring has regularity 
$(3-1)+(4-1)=5,$
so the evaluation map is surjective in degree $7$. Therefore
Theorem~\ref{thm:geproci-lifting} gives a complete intersection
of type $(3,4,8)$. Thus we obtain a geproci set of
$96$ points in $\PP^4$.
\end{example}

\begin{remark}
\label{rem:fixed-characteristic}
The construction takes place over a fixed characteristic. What grows
with the dimension is not the characteristic, but the finite field used
for the Artin–Schreier orbits. More generally, at each lifting step one
may take an orbit of size $p^e$ for any sufficiently large $e$. Thus the
characteristic may be kept fixed throughout the construction, while the
exponent $e$ is allowed to grow. In particular, in every positive
characteristic the construction produces geproci sets in projective
spaces of arbitrary dimension.
\end{remark}

\section{Artin--Schreier rigidity in \texorpdfstring{$\PP^3$}{P3}}\label{sec:rigidity}
Theorem~\ref{thm:P3-characterization} starts with the assumption that the point sets on the supporting lines are Artin--Schreier orbits. It is natural to ask to what extent this structure is forced by the geproci property itself. In the extremal case $q=N$, it is in fact forced. More generally, the number of supporting lines determines how many of the leading coefficients must vanish.

Let $L_1,\ldots,L_q\subset\PP^3$ be distinct concurrent lines spanning $\PP^3$, with common point $O$, and let $B_i\subset L_i\setminus\{O\}$ consist of $N=p^e$ distinct points. As before, write the $i$-th line as $[t:v_i]$, and let
\[
f_i(t)
=
t^N+e_{1,i}t^{N-1}+\cdots+e_{N-1,i}t+e_{N,i}
\]
be the monic polynomial cutting out $B_i$.

\begin{proposition}
\label{prop:forced-sparsity}
Assume $3\le q\le N$, and set
$Z=B_1\cup\cdots\cup B_q.$
Then $Z$ is $(q,N)$-geproci if and only if
$e_{k,i}=0$
for every $i=1,\ldots,q$
and $1\le k\le q-2.$
\end{proposition}
\begin{proof}
We first note that the proof of Proposition~\ref{prop:AS-interpolation} applies without change to arbitrary monic polynomials. More precisely, if the prescribed points on the $i$-th supporting line are cut out by
\[
t^N+c_{1,i}t^{N-1}+\cdots+c_{N,i},
\]
then a degree-$N$ hypersurface containing none of the supporting lines cuts out the prescribed point sets if and only if, for each $k=1,\ldots,N$, the values $c_{k,i}$ are interpolated by a form of degree $k$ in the directions of the supporting lines. Indeed, after normalizing $F(O)=1$, this follows by comparing coefficients on each supporting line. Conversely, one chooses $M_k\in S_k$ with $M_k(v_i)=c_{k,i}$ and takes $F=x_0^N+\sum_{k=1}^N x_0^{N-k}M_k$.

\emph{Necessity.}
Suppose that $Z$ is $(q,N)$-geproci. Let $P=[1:u]$ be a general center of projection, set $\Gamma_P=\pi_P(Z)$, and let $G_P$ denote the union of the $q$ projected supporting lines. Since $\Gamma_P$ is a complete intersection of type $(q,N)$, if $q<N$, then $\dim_K[I(\Gamma_P)]_q=1$, so $G_P$ is, up to scalar, the degree-$q$ generator of the complete intersection. If $q=N$, then $\dim_K[I(\Gamma_P)]_N=2$, and $G_P$ can be completed to a basis by a form $F_P\in[I(\Gamma_P)]_N\setminus KG_P$. Thus, in either case, 
$I(\Gamma_P)=(G_P,F_P)$ 
for some form $F_P$ of degree $N$. In particular, $F_P$ contains none of the projected supporting lines and cuts out on each of them precisely the corresponding $N$ projected points.

Write, as in the proof of Theorem~\ref{thm:P3-characterization}, $v_i=\alpha_i u+w_i$. The projected line is parametrized by $s=\alpha_i-t$. Setting $e_{0,i}=1$, the coefficient of $s^{N-k}$ in $f_i(\alpha_i-s)$ is
\[
(-1)^{N-k}\sum_{r=0}^k
e_{r,i}\binom{N-r}{k-r}\alpha_i^{\,k-r}.
\]
We prove, level by level, that $e_{k,i}=0$ for every $i$ and every $1\le k\le q-2$.

Assume that this has been proved for all levels smaller than $k$. Since $N=p^e$, we have $\binom Nk=0$ in $K$ for $1\le k<N$. The terms with $1\le r< k$ vanish by induction, while the term with $r=k$ is $e_{k,i}$. Hence the coefficient above is, up to the common nonzero scalar $(-1)^{N-k}$, exactly $e_{k,i}$. Applying the interpolation criterion to $F_P$, now on the projected supporting lines, gives
\[
(e_{k,1},\ldots,e_{k,q})\in\operatorname{Im}(\ev_k),
\qquad
\ev_k:H^0(\PP^1,\mathcal O_{\PP^1}(k))\longrightarrow K^q,
\]
where $\ev_k$ is evaluation at the directions of the projected lines.

Fix $i_0$. Since the directions are not all collinear, there exist $a,b$ such that $v_{i_0},v_a,v_b$ are not collinear. Since $k+2\le q$, these three indices can be completed, if necessary, to a subset $I\subset\{1,\ldots,q\}$ of cardinality $k+2$ containing $i_0$, whose directions are therefore not all collinear. Restricting the preceding evaluation condition to $I$, the resulting evaluation map has rank $k+1$, since it is evaluation of degree-$k$ binary forms at $k+2$ distinct points of $\PP^1$. Its image therefore has codimension one. Expanding the corresponding augmented Vandermonde determinant gives
\[
\sum_{i\in I}\varepsilon_i e_{k,i}
\prod_{\substack{h<\ell\\ h,\ell\in I\setminus\{i\}}}
[u,v_h,v_\ell]=0,
\]
where each $\varepsilon_i$ is nonzero. 
 Since this holds for every general $u$, it is a polynomial identity in $u$. Lemma~\ref{lem:Vandermonde-independence}, applied to the directions indexed by $I$, gives $e_{k,i}=0$ for every $i\in I$, and in particular $e_{k,i_0}=0$. Since $i_0$ was arbitrary, the induction is complete.

\emph{Sufficiency.}
Suppose that $e_{k,i}=0$ for every $i$ and $1\le k\le q-2$. After a general projection, the coefficient formula above shows directly that the corresponding coefficients of $s^{N-k}$ vanish for the same values of $k$: the terms with $1\le r\le k$ vanish by hypothesis, while the term with $r=0$ vanishes because $\binom Nk=0$. For every $k\ge q-1$, the evaluation map 
$H^0(\PP^1,\mathcal O_{\PP^1}(k))\longrightarrow K^q$ 
at the $q$ distinct projected directions is surjective. The interpolation criterion therefore gives a degree-$N$ curve $F_P$ containing the projected points and none of the projected supporting lines.

Hence $F_P$ and $G_P$ have no common component. They contain the $qN$ reduced points of $\Gamma_P$, while B\'ezout gives $\deg V(G_P,F_P)=qN$. Therefore $\Gamma_P=V(G_P,F_P)$ scheme-theoretically, and $Z$ is $(q,N)$-geproci.
\end{proof}

The extremal case $q=N$ gives a precise answer to the question above.

\begin{corollary}
\label{cor:AS-forced}
Under the hypotheses of Proposition~\ref{prop:forced-sparsity}, assume $q=N$. Then $Z$ is $(N,N)$-geproci if and only if each $B_i$ is an $\FF_N$-Artin--Schreier orbit.
\end{corollary}

\begin{proof}
If $Z$ is $(N,N)$-geproci, Proposition~\ref{prop:forced-sparsity} gives $e_{1,i}=\cdots=e_{N-2,i}=0$, and hence
$f_i(t)=t^N+e_{N-1,i}t+e_{N,i}.$
Since $f_i$ has $N$ distinct roots, its derivative is the nonzero constant $e_{N-1,i}$. Setting $\mu_i=-e_{N-1,i}$ and $\nu_i=e_{N,i}$, we have $f_i(t)=t^N-\mu_i t+\nu_i$ with $\mu_i\ne0$. Choose a root $\beta_i$ of $f_i$ and $\lambda_i\in K^\times$ such that $\lambda_i^{N-1}=\mu_i$. Then
$f_i(\beta_i+\lambda_i a)
=\lambda_i^N(a^N-a),$
so the roots of $f_i$ are precisely $\beta_i+\lambda_i\FF_N$. Thus $B_i$ is an $\FF_N$-Artin--Schreier orbit. The converse follows from Theorem~\ref{thm:P3-characterization}.
\end{proof}

Thus the Artin--Schreier structure is forced in the extremal case $q=N$, whereas for $q<N$ the situation is more flexible: Proposition~\ref{prop:forced-sparsity} leaves the coefficients
$e_{q-1,i},\ldots,e_{N,i}$
unrestricted. The following example shows that this additional freedom need not come from an additive group action.

\begin{example}
\label{ex:nonadditive-geproci}
Let $N=8$, $q=3$, and work in characteristic $2$. Choose
$\alpha\in\FF_{16},$ such that 
 $\alpha^4+\alpha+1=0,$ 
and set
$$B=
\{
\alpha+\alpha^2,\,
1+\alpha+\alpha^2,\,
\alpha^2,\,
1+\alpha^2,\,
\alpha,\,
1+\alpha,\,
\alpha+\alpha^2+\alpha^3,\,
1+\alpha+\alpha^2+\alpha^3
\}.$$
The eight elements of $B$ are distinct and nonzero, and 
$\sum_{b\in B}b=0.$ 
Hence the monic polynomial with root set $B$ has vanishing $t^7$-coefficient.

On the other hand, $B$ is not a translate of an additive subgroup. Indeed, translating by
$\beta=\alpha+\alpha^2\in B$
gives 
$B-\beta
=\{0, 1, \alpha, 1+\alpha, \alpha^2, 1+\alpha^2, 
\alpha^3, 1+\alpha^3\}.$

This set is not closed under addition: it contains $\alpha$ and $\alpha^2$, but not $\alpha+\alpha^2$.

Now let
$O=[1:0:0:0]$
and take the three coordinate lines through $O$. Define
\[
Z=
\{[b:1:0:0]:b\in B\}
\cup
\{[b:0:1:0]:b\in B\}
\cup
\{[b:0:0:1]:b\in B\}.
\]
Since $q=3$, Proposition~\ref{prop:forced-sparsity} requires only the coefficient of $t^7$ to vanish. Thus $Z$ is a $(3,8)$-geproci set, although the eight points on each supporting line do not form an orbit of a finite additive group acting by translations.
\end{example}

\section{Further directions}
\label{sec:questions}

\subsection{Is the Artin--Schreier structure forced?}
Example~\ref{ex:nonadditive-geproci} shows that the assumption \(q=N\) in Corollary~\ref{cor:AS-forced} is essential. It also shows that, when \(q<N\), additive orbits are only one possible source of geproci configurations. This raises the more general problem of understanding to what extent the Artin--Schreier structure is forced by the geproci property, both in \(\PP^3\) and in higher dimension.

\begin{question}
To what extent is the Artin--Schreier structure forced for geproci sets supported on concurrent lines in \(\PP^n\)? More specifically, can one characterize such sets in higher dimension, or in \(\PP^3\) when the number of supporting lines is strictly smaller than the number of points on each line?
\end{question}

\subsection{Higher-dimensional constructions}

In $\PP^3$, Theorem~\ref{thm:P3-characterization} gives a complete characterization of the Artin--Schreier configurations that are geproci. In higher dimension, the lifting construction of Section~\ref{sec:lifting} provides a sufficient condition, but it is not clear to what extent either the construction or its numerical hypotheses are necessary. This leads to two related questions.

\begin{question}
In Theorem~\ref{thm:P3-characterization}, the condition $q\leq N$ is both necessary and sufficient for an $\FF_N$-Artin--Schreier configuration on $q$ concurrent lines spanning $\PP^3$ to be $(q,N)$-geproci. Is there an analogous characterization in higher dimension?
\end{question}

\begin{question}
How sharp is the regularity bound in the Artin--Schreier lifting construction? More precisely, what is the minimum number of points required on each supporting line for the lifting to work?
\end{question}

\subsection{Exceptional centers and Weddle schemes}

Weddle loci provide a natural way to study exceptional centers of projection; see \cite{POLITUS6} for the definition, history, and recent developments of the subject. In the setting of Theorem~\ref{thm:P3-characterization}, when $q>N$ a general center does not give a complete-intersection projection, but special centers may still do so. Besides components arising from the incidence geometry of the supporting lines, the interpolation criterion suggests that additional components may depend on the Artin--Schreier data. This leads to the following question.

\begin{question}
Let $Z\subset\PP^3$ be an $\FF_N$-Artin--Schreier configuration supported on $q>N$ concurrent lines. What is the $N$-Weddle scheme of $Z$? Can its components, or its reduced support, be described explicitly in terms of the supporting lines and the Artin--Schreier data?
\end{question}

\subsection{Characteristic zero}

The constructions in this paper rely essentially on the special form of Artin--Schreier polynomials in positive characteristic and on its behavior under projection. It remains open whether the geproci phenomenon in higher dimension can occur at all in characteristic zero.

\begin{question}
Does there exist a finite geproci set in $\PP^n$, for some $n>3$, over an algebraically closed field of characteristic zero?
\end{question}

\noindent{\bf Acknowledgments}
The authors used GPT-5.6 Sol (OpenAI) only as an assistive tool for language editing and for checking and improving the presentation of some mathematical arguments throughout the manuscript. The authors take full responsibility for the content of the manuscript.

%grant information

\bibliographystyle{amsplain}

\begin{thebibliography}{99}

\bibitem{ArtinSchreier1927}
E.~Artin and O.~Schreier,
\emph{Algebraische Konstruktion reeller K\"orper},
Abh. Math. Sem. Univ. Hamburg \textbf{5} (1927), 85--99.


\bibitem{POLITUSbook}
L.~Chiantini, {\L}.~Farnik, G.~Favacchio, B.~Harbourne, J.~Migliore,
T.~Szemberg, and J.~Szpond,
\emph{Configurations of points in projective space and their projections}, Preprint 
\href{https://arxiv.org/abs/2209.04820}{arXiv:2209.04820}.

\bibitem{POLITUS4}
L.~Chiantini, {\L}.~Farnik, G.~Favacchio, B.~Harbourne, J.~Migliore,
T.~Szemberg, and J.~Szpond,
\emph{Finite sets of points in $\PP^4$ with special projection properties},
Geom. Dedicata \textbf{219} (2025), no.~2, Paper No.~27, 29 pp.
%\href{https://arxiv.org/abs/2407.01744}{arXiv:2407.01744}.

\bibitem{POLITUS6}
L.~Chiantini, {\L}.~Farnik, G.~Favacchio, B.~Harbourne, J.~Migliore,
T.~Szemberg, and J.~Szpond, \emph{Weddle schemes}. Preprint \href{https://arxiv.org/abs/2606.25060}{arXiv:2606.25060}. 


\bibitem{ChiantiniMigliore2021}
L.~Chiantini and J.~Migliore,
\emph{Sets of points which project to complete intersections, and unexpected
cones},
Trans. Amer. Math. Soc. \textbf{374} (2021), 2581--2607.

\bibitem{FavacchioKettinger2026}
G.~Favacchio and J.~Kettinger,
\emph{Collinearly complete sets and finite subgroups from configurations of
skew $n$-planes in $\PP^{2n+1}_K$}, Preprint
\href{https://arxiv.org/abs/2607.11259}{arXiv:2607.11259} (2026).

\bibitem{Kettinger2024}
J.~Kettinger,
\emph{The geproci property in positive characteristic},
Proc. Amer. Math. Soc. \textbf{152} (2024), 3229--3242.
%\href{https://arxiv.org/abs/2307.04857}{arXiv:2307.04857}.

\bibitem{StacksAS}
The Stacks Project Authors,
\emph{The Stacks Project, Section 9.25: Artin--Schreier extensions},
\href{https://stacks.math.columbia.edu/tag/09DY}{Tag 09DY}.

\end{thebibliography}

\end{document}